\documentclass[11pt]{article}
\usepackage{amsmath,amsthm,amsfonts,amssymb,amscd, amsxtra}
\theoremstyle{plain}
\newtheorem{theorem}{Theorem}[section]

\newtheorem{definition}{Definition}[section]
\newtheorem{corollary}{Corollary}[section]
\newtheorem{proposition}{Proposition}[section]

\newtheorem{remark}{Remark}[section]

\newcommand{\beq}{\begin{equation}}
\newcommand{\eeq}{\end{equation}}
\newcommand{\beqa}{\begin{eqnarray}}
\newcommand{\eeqa}{\end{eqnarray}}
\newcommand{\beqas}{\begin{eqnarray*}}
\newcommand{\eeqas}{\end{eqnarray*}}
\def\min{\operatorname{min}}

\DeclareMathOperator{\grad}{grad}

\begin{document}

\title{Weak Sharp Minima and Finite Termination of the Proximal Point Method for Convex Functions on Hadamard Manifolds}

\author{
Bento, G. C.\thanks{IME, Universidade Federal de Goi\'as,
Goi\^ania, GO 74001-970, BR (Email: {\tt glaydston@mat.ufg.br})}
\and
Cruz Neto, J. X.
\thanks{DM, Universidade Federal do Piau\'{\i},
Teresina, PI 64049-500, BR (Email: {\tt jxavier@ufpi.br}). This author
was partially supported by CNPq GRANT 301625-2008 and PRONEX--Optimization(FAPERJ/CNPq)}
}
% \date{April 13, 2010}

\maketitle
\begin{abstract}
In this paper we proved that the sequence generated by the proximal point method, associated to a unconstrained optimization problem in the Riemannian context, has finite termination when the objective function has a weak sharp minima on the solution set of the problem.
\end{abstract}
\section{Introduction}
Consider the following minimization problem
\begin{eqnarray}\label{po:conv}
\begin{array}{clc}
  (P) & \min  f(p) \\
   & \textnormal{s.t.}\,\,\, p\in M,\\
\end{array}
\end{eqnarray}
where $M$ is a complete Riemannian manifold and $f:M\to\mathbb{R}$ is a function.  For a starting point $p^0\in M$, the exact proximal point method to solve optimization problems of the form \eqref{po:conv} generates a sequence $\{p^k\}\subset M$ as follows:
\begin{equation}\label{eq:prox1}
p^{k+1}\in \mbox{argmin}_{p\in M}\left\{f(p)+\lambda_k d^2(p,p^k)\right\},
\end{equation}
where $\{\lambda_k\}$ is a sequence of positive numbers and $d$ is the Riemannian distance (see Section~\ref{sec2} for a definition). This method was first considered in this context by Ferreira and Oliveira \cite{FO2002}, when $M$ is a Hadamard manifold (see Section~\ref{sec2} for a definition) and $f$ is convex. They proved that, for each $k\in \mathbb{N}$, the function $f+d^2(.\,p^k):M\to\mathbb{R}$ is 1-coercive and, consequently, that the sequence $\{p^k\}$ is well-defined, with $p^{k+1}$ being uniquely determined. Moreover, assuming that $\sum_{k=0}^{+\infty}1/\lambda_k=+\infty$ and that $f$ has a minimizer, the authors proved that the sequence $\{f(x^k)\}$ converges to the minimum value and the sequence $\{x^k\}$ converges to a minimizer point. Li et al.~\cite{ChongLi2009} extended this method for finding singularity of a multivalued vector field and proved that the generated sequence is well-defined and converges to a singularity of a maximal monotone vector field, whenever it exists.

This paper is part of a wider research program consisting of the extension of concepts and techniques of the Mathematical Programming of the Euclidean space $\mathbb{R}^n$ to Riemannian manifolds. It is noteworthy that in a number of recent research papers, several ideas, techniques and algorithms of Euclidean spaces have been extended to Riemannian manifolds and have been used for both theoretically and practical purposes; see \cite{ChongLi2009, Absil2007,Teboulle2004,  Azagra2005, Alvarez2008, BP09, XFLN2006, FS02, FO06, lili09-1,PP08,PO2009, Wang2011-2,wanglihuMa2011,Wang2008,Wang2006-1,Bento2011-2, BFO2010, Bento2012, Tang2011, Wang2011-3, Ferreira2012, Zhu2007, Boris2011} and the references therein. We observe that these extensions allow the solution of some nonconvex constrained problems in Euclidean space. More precisely, nonconvex
problems in the classic sense may become convex with the introduction of an adequate Riemannian metric on the manifold (see, for example \cite{XFLN2006, Bento2011-2}). 

Following the ideas of Ferris \cite{Ferris1991}, we proved in this paper that the sequence generated by the proximal point method associated to the problem \eqref{po:conv} has finite termination when the objective function is convex and the solution set of the Problem \ref{po:conv} is a set of weak sharp minimizers for $f$. As far as we know, the notion of sharp minimizer was introduced by Polyak \cite{Polyak1979} for the case of finite-dimensional Euclidean spaces; see also \cite[page 205]{Polyak1987}. In this particular case it is know that a necessary and sufficient condition for $\bar{p}$ be sharp minimum is that $0\in \mbox{int}\partial f(\bar{p})$.  Rockafellar \cite{Rock1976} showed that,  in a space with linear structure (Hilbert space), this is a sufficient condition for finite termination of the proximal point method. Afterwards, Burke and Ferris \cite{Ferris1993} extended the notion of sharp minima to what became known as weak sharp minima, mainly to include the possibility of multiple solutions, and extended the previous necessary and sufficient condition for characterize the solution set of a minimization problem as a set of weak sharp minimizers. Li et al. \cite{Boris2011} extended the notion of weak sharp minimizer to optimization problems on Riemannian manifolds as well as the previous result which relates finite termination of the proximal point method with weak sharp minima, summarized in Proposition \ref{propo:ms1}.

The organization of our paper is as follows. In Section~\ref{sec2}, some notations and results of Riemannian geometry as well as some fundamental properties and
notations  of convex analysis on Hadamard manifolds, are presented. In Section~\ref{sec:3}, it is  presented the definition of weak sharp minima as well as some of basic related results, and proved the main resulted of the paper. Finally, in Section~\ref{sec4}, we made some last considerations.

\section{Notation and terminology} \label{sec2}

In this section we introduce some fundamental  properties and
notations on Riemannian geometry and convex analysis on Hadamard manifolds which will be
used later.

\subsection{Preliminaries on Riemannian Geometry} 
In this section we introduce some fundamental  properties and
notations on Riemannian geometry. These basics facts can be
found in any introductory book on Riemannian geometry, such as in
\cite{MP92} and \cite{S96}.

Let $M$ be a $n$-dimentional connected manifold. We denote by
$T_pM$  the $n$-dimentional {\it tangent space} of $M$ at $p$, by
$TM=\cup_{p\in M}T_pM$ {\itshape{tangent bundle}} of $M$ and by
${\cal X}(M)$ the space of smooth vector fields over $M$. When $M$
is endowed with a Riemannian metric $\langle \,,\, \rangle$, with
the corresponding norm denoted by $\| \; \|$, then $M$ is now a
Riemannian manifold. We denote by $\mathcal{B}_p:=\{v\in T_pM: \|v\|\leq1\}$ the closed unit  ball of $T _pM$. Recall that the metric can be used to define
the lenght of piecewise smooth curves $\gamma:[a,b]\rightarrow M$
joining $p$ to $q$, i.e., such that $\gamma(a)=p$ and
$\gamma(b)=q$, by
\[
l(\gamma)=\int_a^b\|\gamma^{\prime}(t)\|dt,
\]
and, moreover, by minimizing this length functional over the set
of all such curves, we obtain a Riemannian distance $d(p,q)$ which
induces the original topology on $M$.  Given a nonempty set $D\subset M$, the distance function associated with $D$ is given by 
\[
M\ni p \longmapsto d_D(p):=\inf\{d(q,p): q\in M\}\in \mathbb{R}_+.
\]
The metric induces a map
$f\mapsto\grad f\in{\cal X}(M)$ which associates to each smooth function 
on $M$ its gradient via the rule $\langle\grad
f,X\rangle=d f(X),\ X\in{\cal X}(M)$. Let $\nabla$ be the
Levi-Civita connection associated to $(M,{\langle} \,,\,
{\rangle})$. A vector field $V$ along $\gamma$ is said to be {\it
parallel} if $\nabla_{\gamma^{\prime}} V=0$. If $\gamma^{\prime}$
itself is parallel we say that $\gamma$ is a {\it geodesic}. Given
that geodesic equation $\nabla_{\ \gamma^{\prime}}
\gamma^{\prime}=0$ is a second order nonlinear ordinary
differential equation, then geodesic $\gamma=\gamma _{v}(.,p)$ is
determined by its position $p$ and velocity $v$ at $p$. It is easy
to check that $\|\gamma ^{\prime}\|$ is constant. We say that $
\gamma $ is {\it normalized} if $\| \gamma ^{\prime}\|=1$. The
restriction of a geodesic to a  closed bounded interval is called
a {\it geodesic segment}. A geodesic segment joining $p$ to $q$ in
$ M$ is said to be {\it minimal} if its length equals $d(p,q)$ and
this geodesic is called a {\it minimizing geodesic}. 
A Riemannian manifold is {\it complete} if geodesics are defined
for any values of $t$. Hopf-Rinow's theorem (\cite[Theorem 1.1, page 84]{S96}) asserts that if this
is the case then any pair of points, say $p$ and $q$, in $M$ can
be joined by a (not necessarily unique) minimal geodesic segment.
Moreover, $( M, d)$ is a complete metric space and bounded and
closed subsets are compact. Take $p\in M$. The {\it exponential
map} $exp_{p}:T_{p}  M \to M $ is defined by  $exp_{p}v\,=\,
\gamma _{v}(1,p)$.

%%%%%%
\begin{theorem} \label{T:Hadamard}
Let $M$ be a complete, simply connected Riemannian manifold with nonpositive sectional
curvature. Then $M$ is diffeomorphic to the Euclidean space $\mathbb{R}^n $, $n=dim M $. More
precisely, at any point $p\in M $, the exponential mapping $ exp_{p}:T_{p}M \to M $ is a diffeomorphism.
\end{theorem}
\begin{proof}
See Lemma 3.2 of \cite{MP92}, p. 149 or Theorem 4.1 of \cite{S96}, p. 221.
\end{proof}

A complete simply connected Riemannian manifold of nonpositive
sectional   curvature  is called a {\it{Hadamard manifold}}. The
Theorem \ref{T:Hadamard} says that if $M$ is Hadamard manifold,
then $M$ has the same topology and differential structure of the
Euclidean space $\mathbb{R}^n$. Furthermore,  some
similar geometrical properties of the Euclidean space
$\mathbb{R}^n$ are known, such as, given two points there exists an unique
geodesic that joins them. {\it In this paper, all manifolds $M$
are assumed to be Hadamard finite dimensional}.

Take ${p}\in M $. Let  $ exp^{-1}_{p}:M\to T_{p}M$ be the  inverse
of the  exponential map which is also $C^{\infty }$.  Note that
$d({q}\, , \, p)\,=\,||exp^{-1}_{p}q||$, the map
$d^2(\,.\,,{p})\colon M\to\mathbb{R}$ is  $C^{\infty}$ and
\[
\grad \frac{1}{2}d^2(q,{p})=-exp^{-1}_{q}{p},
\]
see, for example, Proposition 4.8 of \cite{S96} page 108.

\subsection{Convexity on Hadamard manifold} \label{sec3}
In this section, we introduce some fundamental properties and
notations  of convex analysis on Hadamard manifolds that will be
used later. These properties can be
found, for instance, in \cite{U94, RAP97, XFL2000,  ChongLi2009}.

A set $\Omega\subset M$ is said to be {\it convex \/} if any
geodesic segment with end points in $\Omega$ is contained in
$\Omega$. A function $f:M\to\mathbb{R}$ is said to be {\it convex} if for any geodesic segment $\gamma:[a, b]\to M$ the composition $f\circ\gamma:[a,
b]\to\mathbb{R}$ is convex. Take $p\in M$. A vector $s \in T_pM$ is said
to be a {\it subgradient\/} of $f$ at $p$ if
\[
f(q) \geq f(p) + \langle s,\exp^{-1}_pq\rangle,
\]
for any $q\in M$. The set of all subgradients of $f$ at $p$,
denoted  by $\partial f(p)$, is called the {\it subdifferential\/}
of $f$ at $p$. It is known that if $f$ is convex then $\partial f(p)$ is a set non-empty, convex and compact, for each $p\in M$. In particular, given $p,q\in M$ and $u\in \partial f(p)$, $v\in \partial f(q)$, we have
\begin{equation}
\langle u,\exp^{-1}_pq\rangle\leq f(q)-f(p)\leq \langle v,-\exp^{-1}_qp\rangle.
\end{equation}
But this tell us, from next definition, that if $f$ is convex then $\partial f$ is a monotone vector fields on $M$.
\begin{definition} \label{def:mps}
Let $\Omega\subset M$ be an open convex set and $X$ a point-set vector fields on $M$. $X$ is said to be {\it monotone\/} on $\Omega$, if
\[
\langle u,\exp^{-1}_pq\rangle\leq \langle v,-\exp^{-1}_qp\rangle,\qquad p,q\in \Omega, \quad u\in X(p), \quad v\in X(q).
\]
\end{definition}
\begin{remark}
This last definition has appeared in \cite{ChongLi2009}, but also it is worth to point out that an equivalent definition has appeared in \cite{XFL2000}.
\end{remark}
Let $C\subset M$ be a set nonempty, convex and closed. It is well-known (see Corollary 3.1 of \cite{FO2002}) that for each $p\in M$ there exists a unique element $\tilde{p}\in C$ such that 
\[
\langle \exp^{-1}_{\tilde{p}}p,\exp^{-1}_{\tilde{p}}q\rangle\leq 0,\qquad q\in C.
\]
In this case, $\tilde{p}$ is the projection of $p$ onto the set $C$ which we will denote by $\Pi_C(p)$.

Let $D\subset\mathbb{R}^n$ be a convex set, and $p\in D$. Following \cite{ChongLi2009}, we define the normal cone to $D$ at $p$ by:
\[
N_D(p):=\{w\in T_pM:\langle w, \exp_{p}^{-1}q \rangle\leq 0, q\in D\}.
\]
The previous definition holds just when $M$ is of the Hadamard type. A more general definition has appeared in \cite{Boris2011}.

\section{Proximal Point and Weak Sharp Minima on Riemannian Manifolds}\label{sec:3}
The definition of weak minima sharp as well as some of the basic related results that are used in this paper were introduced, in the Riemannian context, by Li et al.~\cite{Boris2011} for constrained optimization problems in the Riemannian context.  Assuming that the solution set of   Problem \eqref{po:conv} is a set of weak sharp minimizers associated to the problem in question, we proved that the proximal point method \eqref{eq:prox1} has finite termination. We recall first some basic facts on the sequence generated by proximal point method \eqref{eq:prox1} (see, for instance, \cite{FO2002}).  In the remainder of this paper $ f: M\to\mathbb{R}$ represent a convex function, $U$ denote the solution set of Problem~\ref{po:conv} and $\{p^k\}$ is the sequence generated by \eqref{eq:prox1}. Moreover, we suppose that $U$ is nonempty and closed in $M$. 

Assuming that the sequence $\{p^k\}$ is well defined, it follows that
\[
(\lambda_k/2)d^2(p^{k+1},p^k)\leq f(p^k)-f(p^{k+1}).
\]
Hence, if $f$ is bounded below, then 
\begin{equation}\label{eq:1-1}
\sum_{k=0}^{+\infty}(\lambda_k/2)d^2(p^{k+1},p^k)<+\infty.
\end{equation}
The following proposition gathers some of the main results of \cite{FO2002} associated  to the sequence $\{p^k\}$.
\begin{proposition}\label{propo:tf1}
If $M$ is a Hadamard manifold, then the following statements hold: 
\begin{itemize}
\item [i)] $\{p^k\}$ is well defined and is characterized by
\begin{equation}
\lambda_k (exp_{p_{k+1}}^{-1}p_k)\in \partial f(p_{k+1});
\end{equation}
\item [ii)] If $\sum_{k=0}^{+\infty}1/\lambda_k=+\infty$ and the solution set of the problem \eqref{po:conv} is nonempty, then the sequence $\{p^k\}$ converges to a solution of the problem \eqref{po:conv}.
\end{itemize}
\end{proposition}
Note that if $\lambda_{-}, \lambda_{+}\in \mathbb{R}$ are such that $0<\lambda_{-}\leq\lambda_{k}\leq \lambda_+$, $k\in\mathbb{N}$, in particular $\sum_{k=0}^{+\infty}1/\lambda_k=+\infty$. Assuming that $0<\lambda_{-}\leq\lambda_{k}\leq \lambda_+$, $k\in\mathbb{N}$ and taking in consideration that $\|\exp^{-1}_qp\|=d(q,p)$ (see Section~\ref{sec3}), from \eqref{eq:1-1} combined with  item $i)$ of the previous proposition, it follows that
\begin{corollary}\label{propo:tf2} 
If $\{(p^{k},g^k)\}\in TM$ ($TM$ denote the tangent boundle, see Section~\ref{sec2}) is a sequence such that  $g^{k+1}=\lambda_k (exp_{p_{k+1}}^{-1}p_k)\in \partial f(p_{k+1})$, then
\[
\sum_{k=0}^{+\infty}\|g^k\|<+\infty.
\]
In particular, for each $\alpha>0$ there exists $k_0\in\mathbb{N}$ such that $\|g^{k_0}\|<\alpha$.
\end{corollary}

\begin{definition}
The set $U$ is said be the set of weak sharp minimizers for Problem~\ref{po:conv} with modulus $\alpha>0$ if
\[
f(q)\geq f(p)+\alpha d_{U}(q), \qquad p\in U, \quad q\in M.
\]
\end{definition}

Next proposition establish one characterization of the set of weak sharp minima on Riemannian manifolds.
\begin{proposition}\label{propo:ms1}
A necessary and sufficient condition for $U$ be the set of weak sharp minima for Problem~\ref{po:conv} with modulus $\alpha>0$ is that
\[
\alpha\mathcal{B}_p\cap N_{U}(p)\subset\partial f(p), \qquad p\in U.
\] 
\end{proposition} 
\begin{proof}
Note that $N_M(p)=\{0\}$ and $\partial f(p)$ is a closed set, for each $p\in M$. Hence, the closure of the set $\partial f(p)+N_M(p)$ is equal to $\partial f(p)$, and the proof of this proposition follows from Theorem 4.6 and 5.5 of \cite{Boris2011}.  
\end{proof}
The following proposition is the key of our paper, since it is fundamental in the proof of the central result of this paper. 
\begin{proposition}\label{propo:tf3}
Suppose that $U$ is the set of weak sharp minima for Problem~\ref{po:conv} with modulus $\alpha>0$. If there exist $q\in M$ and $w\in T_qM$ such that $\|w\|< \alpha$ and $w\in\partial f(q)$, then $q\in U$.
\end{proposition}
\begin{proof}
Let us suppose, by contraction, that $q\notin U$ and define 
\[
\tilde{q}=\alpha\frac{\exp^{-1}_{p}q}{d(q,p)},\qquad p:=\Pi_U(q).
\]
From the definition of $p$, it follows that
\[
\langle \exp^{-1}_{p}q,\exp^{-1}_{p}\bar{q}\rangle\leq 0,\qquad \bar{q}\in U.
\]
Combining definition of $\tilde{q}$ with last inequality and definition of the normal cone $N_U(p)$, we obtain $\tilde{q}\in N_U(p)$.
Hence, since $\tilde{q}\in \alpha\mathcal{B}_{\tilde{q}}$, we conclude that 
\[
\tilde{q}\in   \alpha\mathcal{B}_p\cap N_{U}(p).
\]
Now, using that $U$ is the set of weak sharp minima for Problem~\ref{po:conv} with modulus $\alpha>0$, $p\in U$ and Proposition~\ref{propo:ms1}, last inclusion implies that 
$\tilde{q}\in \partial f(p)$. Recall that $\partial f$ is a monotone field. So, since $w\in \partial f(q)$ and $\tilde{q}\in \partial f(p)$, we have
\begin{equation}\label{eq:cs1}
\langle \tilde{q},\exp^{-1}_pq\rangle\leq \langle w,-\exp^{-1}_qp\rangle.
\end{equation}
Taking account that $\langle w,-\exp^{-1}_qp\rangle\leq \|w\|\|\exp^{-1}_qp\|$, $\|\exp^{-1}_qp\|=d(q,p)$ and definition of $\tilde{q}$, from inequality \eqref{eq:cs1}, we obtain
\[
\|w\|\geq \alpha,
\]
which contradicts the inequality $\|w\|<\alpha$, and the desired result is proved.
\end{proof}
Next we present the central result of this paper.
\begin{theorem}
Suppose that $U$ is the set of weak sharp minima for Problem~\ref{po:conv} with modulus $\alpha>0$ and let $p^0\in\mathbb{R}^n$. If $\{\lambda_k\}$ is a sequence of real numbers and $\lambda_{-},\lambda_{+}$ positive constants such that $\lambda_{-}\leq\lambda_{k}\leq \lambda_+$, $k\in\mathbb{N}$, then the proximal point method terminates in a finite number of iterations.
\end{theorem}
\begin{proof}
The proof follows immediately from Proposition~\ref{propo:tf1} combining with Corollary~\ref{propo:tf2} and Proposition~\ref{propo:tf3}.
\end{proof}
%%%%%%%%%%%%%%%%%%%%%%%%%%%%%%%%%%%%%%%%%%%%%%%%%%%%%%%%%%%%%%%%%%%%%%%%%%%
\section{Final Remarks}\label{sec4}
In this paper we recall the notion of weak sharp minima for unconstrained optimization problem on Riemannian manifolds and we explored properties of weak sharp minimum on Hadamard manifold to establish the finite termination of the proximal point method.

%%%%%%%%%%%%%%%%%%%%%%%%%%%%%%%%%%%%%%%%%%%%%%%%%%%%%%%%%%%%%%%%%%%%%%%%%%

%%%%%%%%%%%%%%%%%%%%%%%%%%
%%%%%%%%%%%%%%%%%%%%%%%%%%

\end{document}